\documentclass{article}

\usepackage[utf8]{inputenc}
\usepackage[T1]{fontenc}
\usepackage{lmodern}
\usepackage{amsmath}
\usepackage{amsthm}
\usepackage{amssymb}
\usepackage{mathrsfs}
\usepackage{graphicx}
\usepackage{url}
\usepackage{hyperref}
\usepackage{enumitem}
\usepackage[a4paper, hmargin=3cm, vmargin=3cm]{geometry}
\usepackage[nottoc, notlof, notlot]{tocbibind} 

\newcommand{\ud}{\,\mathrm{d}}

\newcommand{\R}{\mathbb{R}}

\newcommand{\la}{\left\langle}
\newcommand{\ra}{\right\rangle}
\newcommand{\longueur}{\textup{length}}

\newtheorem{theoreme}{Théorème}[section]
\newtheorem{lemme}[theoreme]{Lemma}

\newtheorem{prop}[theoreme]{Proposition}
\newtheorem{prop-def}[theoreme]{Definition-proposition}

\theoremstyle{definition}

\theoremstyle{remark}

\newtheorem{exs}[theoreme]{Examples}
\newtheorem{rmq}[theoreme]{Remark}

\title{Improvement and generalisation of Papasoglu's lemma}
\author{\textsc{Simon Allais}}

\begin{document}

\maketitle
\begin{abstract}
We improve an isoperimetric inequality due to Panos Papasoglu. We also generalize this inequality to the Finsler case by proving an optimal Finsler version of the Besicovitch's lemma which holds for any notion of Finsler volume.
\end{abstract}

\section{Introduction}
In \cite{Pa} (proposition 2.3), Panos Papasoglu shows the
\begin{lemme}
Let $(\mathbb{S}^{2},g)$ be a Riemannian two-sphere and denote by $\mathcal{A}$ its Riemannian area. Then for any $\varepsilon>0$ there exists a closed curve $\gamma$ dividing $(\mathbb{S}^{2},g)$ into two disks $D_{1}$ and $D_{2}$ of area at least $\frac{\mathcal{A}(\mathbb{S}^{2},g)}{4}$ and whose length satisfies
\begin{equation*}
\longueur (\gamma ) \leq 2\sqrt{3\mathcal{A}(\mathbb{S}^{2},g)} + \varepsilon.
\end{equation*}
\end{lemme}

This lemma has several deep consequences in metric geometry: using it, P. Papasoglu gives estimates of the Cheeger constant of surfaces, Y. Liokumovich, A. Nabutovsky and R. Rotman use it to answer a question asked by S. Frankel, M. Katz and M. Gromov in \cite{Lio1} whereas F. Balacheff uses it to estimate 2-spheres width in \cite{Bal2}. In this article, we give two different ways to improve Papasoglu estimate. First by a $\sqrt{2}$ factor by using directly the coarea formula instead of the Besicovitch lemma. Then by a $2\sqrt{\frac{2}{\pi}}$ factor by using an argument suggested by an anonymous reviewer and already used by Gromov to give the filling radius of $\mathbb{S}^{1}$ in the simply connected case: Pu's inequality. It gives automatically better estimates: for instance, in \cite{Lio2}, the constants 52 and 26, given by Y. Liokumovich in the abstract, could be divided by $2\sqrt{\frac{2}{\pi}}$, thus \emph{there exists a Morse function $f: M \rightarrow \R$, which is constant on each connected component of a Riemannian 2-sphere with $k\geq 0$ holes $M$ and has fibers of length no more than $26\sqrt{\frac{\pi}{2}\mathcal{A}(M)}$} and \emph{on every 2-sphere there exists a simple closed curve of length $\leq 13\sqrt{\frac{\pi}{2}\mathcal{A}(\mathbb{S}^{2})}$ subdividing the sphere into two discs of area $\geq \frac{1}{3}\mathcal{A}(\mathbb{S}^{2})$}.

The Besicovitch's lemma asserts that, given a parallelotope $P\subset \R^{n}$ endowed with a Riemannian metric $g$ then
\begin{equation*}
v(P,g) \geq \prod_{i=1}^{n}d_{i}
\end{equation*}
where $v$ denotes the Riemannian volume of $(P,g)$ and the $d_{i}$ denote the Riemannian distances between two opposite sides of $P$ (see for instance \cite{Gr} section 4.28). It was used by P. Papasoglu in the proof of his lemma. In this article, we give a natural generalisation of Besicovitch's lemma extending it to Finsler parallelotopes -- that is parallelotopes continously endowed with a norm at each of their points. As for such a manifold, there aren't one good definition of volume, we prove an optimal inequality satisfied by any Finsler volume in the sense of \cite{BBI} (paragraph 5.5.3) such as Busemann-Hausdorff and Holmes-Thompson ones. Our proof is based on the Gromov one given in \cite{Gr}. We then use it in order to extend the Papasoglu lemma to Finsler 2-spheres although the Holmes-Thompson and Busemann-Hausdorff cases could still be improved.

\section{Improvements of Papasoglu's lemma}

The Riemannian case of Papasoglu isoperimetric inequality could be improved by using directly the coarea formula instead of Besicovitch lemma:
\begin{prop} \label{papRiem}
Let $(\mathbb{S}^{2},g)$ be a Riemannian two-sphere and denote by $\mathcal{A}$ its Riemannian area. Then for any $\varepsilon>0$ there exists a closed curve $\gamma$ dividing $(\mathbb{S}^{2},g)$ into two disks $D_{1}$ and $D_{2}$ of area at least $\frac{\mathcal{A}(\mathbb{S}^{2},g)}{4}$ and whose length satisfies
\begin{equation*}
\longueur (\gamma ) \leq \sqrt{6\mathcal{A}(\mathbb{S}^{2},g)} + \varepsilon.
\end{equation*}
\end{prop}

\begin{proof}
Let $\Gamma$ be the set of simple closed curves dividing $(\mathbb{S}^{2},g)$ into two disks of area $\geq\frac{\mathcal{A}(\mathbb{S}^{2},g)}{4}$. Let $L=\inf_{\gamma\in \Gamma}\longueur (\gamma)$. Now if we fix an $\varepsilon > 0$, we can take $\gamma\in U$ such as $\longueur (\gamma) < L+\varepsilon$ and denote by $D_{1}$ and $D_{2}$ the two disks bounded by $\gamma$ with $\mathcal{A}(D_{1})\geq \mathcal{A}(D_{2})$ (which implies $\mathcal{A}(D_{1})\geq \frac{\mathcal{A} (\mathbb{S}^{2},g)}{2}$).

Then $\gamma$ cannot be $\varepsilon$-shortcuts on $D_{1}$ -- that is there doesn't exist any $\delta\subset D_{1}$ joigning two points $a$ and $b$ of $\gamma$ of length $\longueur (\delta) < \longueur (\gamma_{1}) - \varepsilon$ where $\gamma_{1}\subset \gamma$ is the shortest curve between $a$ and $b$ on $\gamma$. In the contrary then either $\delta\cup\gamma_{1}$ or $\delta\cup\gamma_{2}$ would bound a disk of area $\geq \frac{\mathcal{A} (\mathbb{S}^{2},g)}{4}$ with a length $< L$ (calling $\gamma_{2} = \gamma\setminus\gamma_{1}$), contradiction.

\begin{figure}
\begin{center}
\includegraphics[scale =0.8]{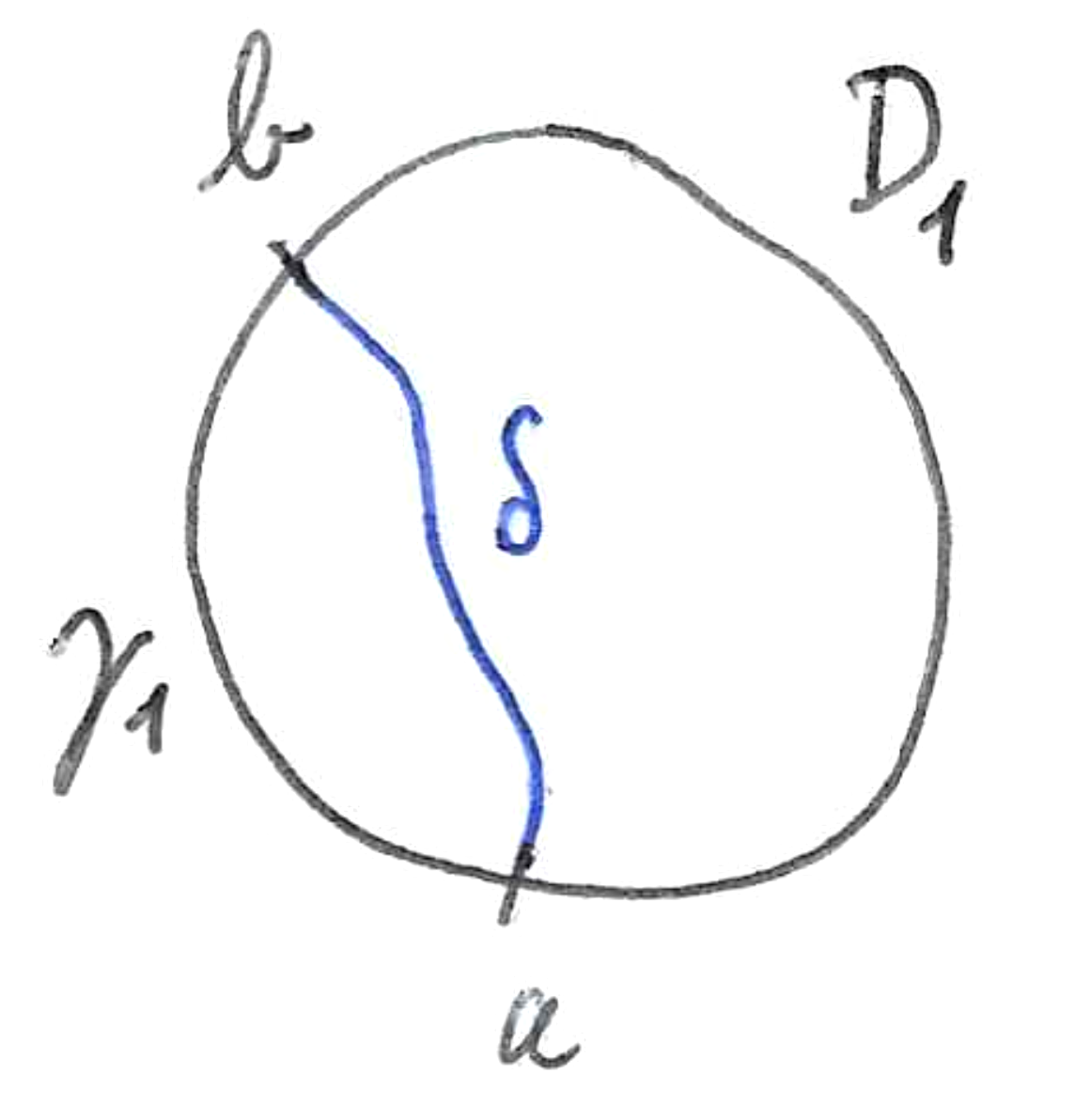}
\caption{$\delta$ can't be an $\varepsilon$-shortcut}
\label{fig:raccourci} 
\end{center}
\end{figure}

Now fix $\varepsilon > 0$ and rather take $\gamma\in U$ a curve of length $\longueur (\gamma)<L+\frac{\varepsilon}{L}$ (taking $\varepsilon$ small enough to have $\longueur (\gamma)<2L$). On $D_{1}$ the disk of greatest area, there isn't any $\frac{\varepsilon}{L}$-shortcut between two points of $\gamma$. Fix any point $A$ of $\gamma$ and denote for every $r\geq 0$ $F_{r}:=\{ m\in \overline{D_{1}}\ |\ d(A,m)=r \}$. As $d(A,\cdot)$ is a Lipschitz continuous function, it is differentiable almost everywhere and, according to Sard's lemma, $F_{r}$ is a submanifold for almost every $r$; we will restrict ourselves to such $r$. Let $u(r)$ and $v(r)$ be the two points of $\gamma$ away from $r$ from $A$ when $r<\longueur(\gamma)$. As $F_{r}$ is a submanifold on $D_{1}$ which is a submanifold with bondary on $D_{1}\cup\gamma$, there is a path $\delta_{r}$ of $F_{r}$ connecting $u(r)$ to $v(r)$. Then, as $\frac{\varepsilon}{L}$-shortcuts don't exist,
 \begin{equation*}
  \longueur (\delta_{r}) \geq \left\{
  \begin{array}{c c c}
   2r-\frac{\varepsilon}{L} & \text{ if }& r\leq \frac{\longueur (\gamma)}{4}\\
   \longueur (\gamma) - 2r-\frac{\varepsilon}{L}  & \text{ if }&  \frac{\longueur (\gamma)}{4}\leq r\leq \frac{\longueur (\gamma)}{2}.
  \end{array}
\right.
 \end{equation*}
 
 \begin{figure}
\begin{center}
\includegraphics[scale =0.8]{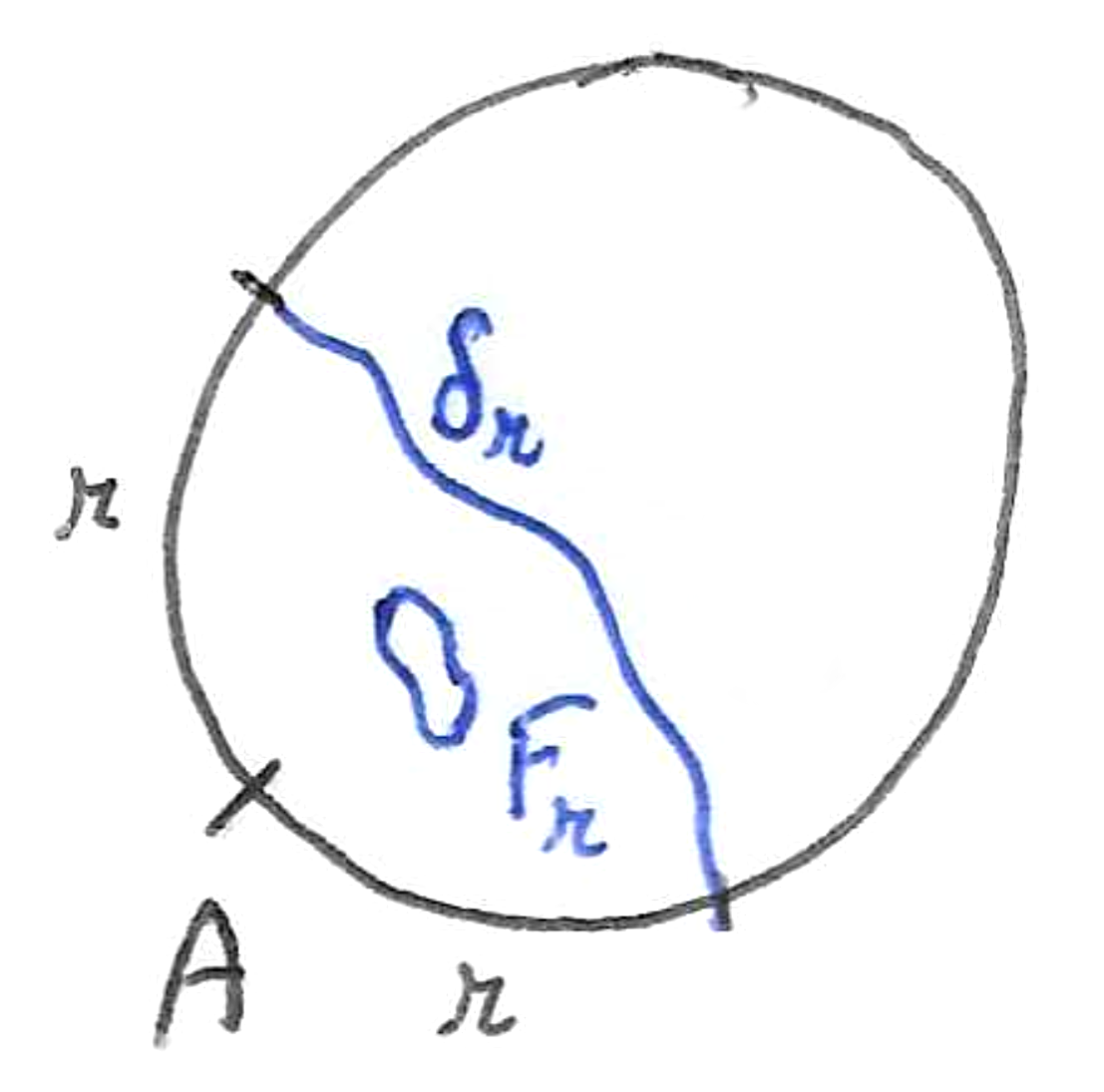}
\caption{The loci $F_{r}$ and $\delta_{r}$}
\label{fig:coaire} 
\end{center}
\end{figure}
 
 Then, using the coarea formula:
 \begin{eqnarray*}
  \mathcal{A}(D_{1}) &=& \int_{0}^{+\infty} \longueur (F_{r})\ud r \\
   &\geq& \int_{0}^{\frac{\longueur(\gamma)}{2}} \longueur (\delta_{r})\ud r \\
   &\geq& 2\int_{0}^{\frac{\longueur(\gamma)}{4}} \left(2r-\frac{\varepsilon}{L}\right)\ud r \\
   &\geq& \frac{\longueur(\gamma)^{2}}{8} -\varepsilon.
 \end{eqnarray*}

 In addition to the fact that $\frac{3}{4}\mathcal{A}(\mathbb{S}^{2})\geq \mathcal{A}(D_{1})$ and that this inequality holds for any $\varepsilon >0$ short enough, we can conclude.

\end{proof}

An anonymous reviewer suggested another way to improve it, using an argument given by Gromov in \cite{fill} (section 5.5.B', item (e)):
\begin{prop} \label{papRiem2}
Let $(\mathbb{S}^{2},g)$ be a Riemannian two-sphere and denote by $\mathcal{A}$ its Riemannian area. Then for any $\varepsilon>0$ there exists a closed curve $\gamma$ dividing $(\mathbb{S}^{2},g)$ into two disks $D_{1}$ and $D_{2}$ of area at least $\frac{\mathcal{A}(\mathbb{S}^{2},g)}{4}$ and whose length satisfies
\begin{equation*}
\longueur (\gamma ) \leq \sqrt{\frac{3\pi}{2}\mathcal{A}(\mathbb{S}^{2},g)} + \varepsilon.
\end{equation*}
\end{prop}

\begin{proof}
Lets having the same approach as the previous proof, taking $\gamma\in \Gamma$ a curve of length $\longueur (\gamma)<L+\varepsilon$ dividing $\mathbb{S}^{2}$ on two disk $D_{1}$ and $D_{2}$ with the same conditions.

As there is no $\varepsilon$-shortcut, any curve joining two antipodal points of $\partial D_{1}$ is longer than $\frac{\longueur (\gamma )}{2} - \varepsilon$. By identification of these antipodal points, $D_{1}$ gives a projective plane of systole greater than $\frac{\longueur (\gamma )}{2} - \varepsilon$, thus, applying Pu's systolic inequality,
\begin{equation*}
\mathcal{A}(D_{1}) \geq \frac{2}{\pi}\left(\frac{\longueur (\gamma )}{2} - \varepsilon\right)^{2}.
\end{equation*}
As $\frac{3}{4}\mathcal{A}(\mathbb{S}^{2})\geq \mathcal{A}(D_{1})$, we then conclude.
\end{proof}

\begin{rmq} \label{optimal}
The equality case of Pu's theorem tells us about the (un)optimality of this inequality. Precisely, there is no riemannian 2-sphere $(\mathbb{S}^{2},g)$ whose minimal closed curve $\gamma$ satisfying Papasoglu's hypothesis has length:
\begin{equation*}
\longueur (\gamma ) = \sqrt{\frac{3\pi}{2}\mathcal{A}(\mathbb{S}^{2},g)}.
\end{equation*}
As a matter of fact, this would imply equality cases $\frac{3}{4}\mathcal{A}(\mathbb{S}^{2},g) = \mathcal{A}(D_{1})$ and $\mathcal{A}(D_{1}) = \frac{2}{\pi}\left(\frac{\longueur (\gamma )}{2}\right)^{2}$. By Pu's theorem, $D_{1}$ is then hemisphere of the round sphere of radius $\frac{\longueur (\gamma )}{2\pi}$.
Let see that $D_{1}$ hemisphere of 2-sphere of radius $r$ implies that $(\mathcal{S}^{2},g)$ is the round sphere of radius $r$. This will conclude because $\gamma$ would be an equator which is obviously not minimal for Papasoglu's lemma.

In order to prove it, we will apply Pu's theorem to $D_{2}$, thus $D_{2}$ would be a round hemisphere of radius $r$. Let see that any curve $\delta_{2}$ joining two antipodal points $N$ and $S$ of $D_{2}$ is longer than $\frac{\longueur (\gamma )}{2} = \pi r$. Suppose the contrary for some $\delta_{2}$ in $D_{2}$ joining $N$ and $S$, then, gluing this curve with any meridian $\delta_{1}$ of the hemisphere $D_{1}$ joining $N$ and $S$, we obtain a close simple curve $\delta = \delta_{1}\cdot\delta_{2}$. As meridians of a round hemisphere of radius $r$ have length $\pi r$, $\longueur (\delta) < \longueur (\gamma)$. But, according to the intermediate value theorem, there exists a meridian $\delta_{1}$ such as $\delta$ divides $(\mathbb{S}^{2},g)$ into disks of same area, a contradiction with $\gamma$ minimality.
\begin{figure}
\begin{center}
\includegraphics[scale =0.7]{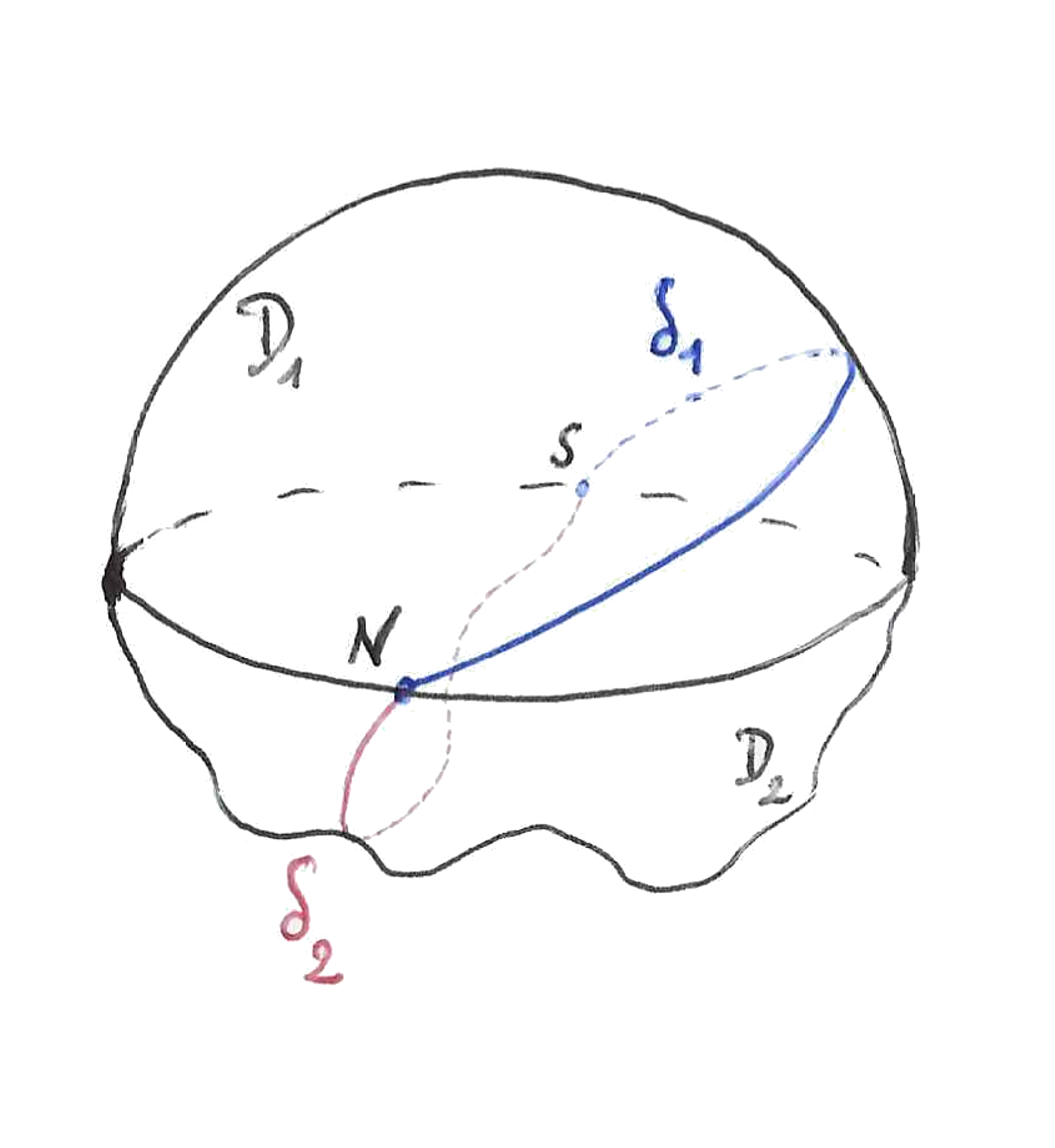}
\caption{$\delta_{2}$ glued with a meridian $\delta_{1}$}
\label{fig:unoptimal} 
\end{center}
\end{figure}
\end{rmq}

\section{Besicovitch's lemma for Finsler manifolds}
In this section, we extend Papasoglu's lemma to Finsler manifolds for any good notion of Finsler area. For this, we first give a natural generalisation of the Besicovitch lemma.

\subsection{Length metric and volume on a Finsler manifolds}

The manifolds used here will be closed and connected. See \cite{BBI} for details and motivations about the results of this section.

Recall that a \emph{continuous Finsler metric} on a manifold $M$ is a continuous function $\Phi : TM\rightarrow [0,+\infty[$ whose restriction to every tangent space is an asymmetric norm. Such a manifold $M$ is said to be a \emph{Finsler manifold} $(M,\Phi)$. If $\Phi(-v_{x}) = \Phi(v_{x})$ for all tangent vectors, we shall say that $\Phi$ is a \emph{reversible} coninuous Finsler metric.

We can then define a \emph{length metric} $d_{\Phi}$ on $M$ by:
\begin{equation*}
\forall x,y\in M, \quad d_{\Phi}(x,y) = \inf_{\gamma : x\leadsto y} \longueur_{\Phi}(\gamma)
\end{equation*}
where the infimum is taken on the piecewise-$\mathcal{C}^{1}$ curves $\gamma: [0,1]\rightarrow M$ joining $x$ to $y$ and
\begin{equation*}
\longueur_{\Phi}(\gamma) := \int_{0}^{1}{\Phi(\gamma'(t))\ud t}.
\end{equation*}

We will restrict ourselves to the case of reversible continuous Finsler metrics.

Contrary to the Riemannian case, there isn't one natural way to define a volume on Finsler manifolds. We will give two natural definitions. The Busemann-Hausdorff volume could be defined, for all open subset $U\subset (M,\Phi)$ as:
\begin{equation*}
 v_{BH}(U) := \int_{U} \frac{|B_{g}|}{|B_{\Phi}|}v_{g}
\end{equation*}
where $v_{g}$ is the volume associated to $g$ which is a Riemannian auxiliary metric, for every $p\in M$, $|A|$ designated the $g_{p}$-normalised Lebesgue measure of $A\subset T_{p}M$ and $B_{g}$ and $B_{\Phi}$ are unit balls of $T_{p}M$ endowed with the normed metrics $g_{p}$ and $\Phi_{p}$ respectively. This definition does not depend on $g$ and boils done to normalise the volume of the unit ball of each tangent space $(T_{p}M,\Phi_{p})$.

The Holmes-Thompson volume is defined, for all open subset $U\subset (M,\Phi)$ as:
\begin{equation*}
 v_{BH}(U) := \int_{U} \frac{|B_{\Phi}^{*}|}{|B_{g}|}v_{g}
\end{equation*}
where, for all $p\in M$ and all convex $K$ of the euclidean space $(T_{p}M,g_{p})$, $K^{*} := \{ u \in T_{p}M \ |\ \forall w\in K,\ g_{p}(u,w) \leq 1 \}$ is the dual convex of $K$. Compared to the Busemann-Hausdorff volume, here we normalise the unit dual ball.

In the case of a Riemannian manifold $(M,g)$, Busemann-Hausdorff and Holmes-Thompson volumes are equal and $v_{BH}=v_{HT}=v_{g}$.

These two volumes are \emph{monotonous}: $v$ is monotonous if for all short application between Finsler manifolds $f:(M,\Phi) \rightarrow (N,\Psi)$,
\begin{equation*}v(f(M),\Psi)\leq v(M,\Phi).\end{equation*}

We refer to paragraph 5.5.3 of \cite{BBI} for an in-depth analysis of the general notion of Finsler volumes (which includes these two). 

\subsection{Finsler Besicovitch's Lemma}

We show that we can deduce a more general statement of Besicovitch's lemma from the proof given in section 4.28 of \cite{Gr}:

\begin{prop}[Finsler Besicovitch's lemma] \label{prop:besicovitch}
Let $P\subset \R^{n}$ be a $n$-dimensionnal parallelotope endowed with a reversible continuous Finsler metric $\Phi$. If $(F_i,G_i)$ (with $1\leq i\leq n$) denotes its pairs of opposite faces and $d_{i} := d_{\Phi}(F_{i},G_{i})$, then, for any Finsler volume $v$,
\begin{equation*}
v(P,\Phi) \geq v\left(\prod_{i=1}^{n}[0,d_{i}],\|\cdot\|_{\infty}\right).
\end{equation*}
\end{prop}

\begin{proof}
Let $f$ be the continous function
\begin{equation*}
f:\left\{
\begin{array}{c c c}
P &\rightarrow& \R^{n}\\
x &\mapsto& \left(d_{\Phi}(x,F_{i})\right)_{1\leq i\leq n}
\end{array}\right. .
\end{equation*}
As for all points $x$ and $y$ in $P$,
\begin{equation} \label{ineqbes}
|d_{\Phi}(x,F_{i})-d_{\Phi}(y,F_{i})|\leq d_{\Phi}(x,y)
\end{equation}
for all $i$, considering the maximum among $i$, one has that $f:(P,\Phi)\rightarrow (\R^{n},\|\cdot\|_{\infty})$ is short. Thus, proving $f(P)\supset \prod_{i=1}^{n}[0,d_{i}]=:C$ is enough  to obtain the inequality.

Note that the boundary of $P$ is mapped outside the interior of $C$, more precisely, writing $f=(f^{1},\ldots ,f^{n})$, $f^{i}(F_{i}) = 0$ whereas $f^{i}(G_{i})\subset [d_{i},+\infty[$. From the definition of $P$, there exists an homeomorphism $h:P\rightarrow C$ mapping each face onto a face (with the obvious choice). So $f_{t} = tf_{|\partial P}+(1-t)h_{|\partial P}$ defines a homotopy from $h_{|\partial P}$ to $f_{|\partial P}$ with values in $\R^{n}\setminus \overset{\circ}{C}$. If there exists $y\in \overset{\circ}{C}\setminus f(P)$, then $f$ should be homotopic to $0$ in $\R^{n}\setminus y \supset \R^{n}\setminus \overset{\circ}{C}$, so $h_{|\partial P}$ should also be homotopic to $0$ in $\R^{n}\setminus y$, a contradiction ($h(P)\ni y$).

\begin{figure}
\begin{center}
\includegraphics[scale =0.7]{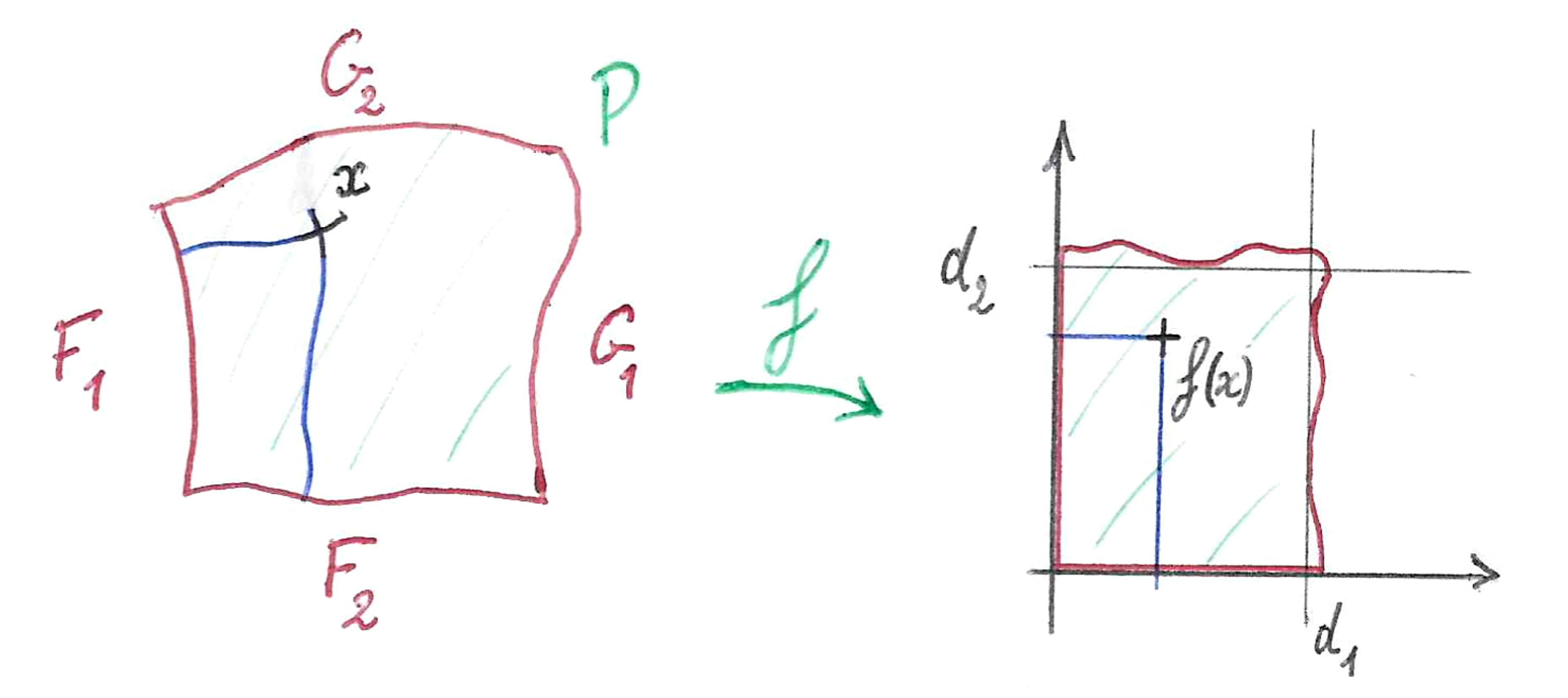}
\caption{Scheme of the proof}
\label{fig:besicovitch} 
\end{center}
\end{figure}

As $v$ is monotonous and $f$ is short, one has the chain of inequalities
\begin{equation*}
v(P,\Phi)\geq v\left(f(P),\|\cdot\|_{\infty}\right)\geq v(C,\|\cdot\|_{\infty}).
\end{equation*}

\end{proof}

\begin{rmq}
The proof provides us some information about the equality case. As $f(P)\supset C$, in order to have $v\left(f(P),\|\cdot\|_{\infty}\right)= v(C,\|\cdot\|_{\infty})$, $f(P)$ and $C$ must only differ from a negligible set of $\R^{n}$. As $f:(P,\Phi)\rightarrow (\R^{n},\|\cdot\|_{\infty})$ is short, in order to have $v(P,\Phi)=v(f(P),\|\cdot\|_{\infty})$, $f$ needs to be locally isometric almost everywhere -- meaning that $\ud f_{x}$, which is defined for almost every $x$, has norm $1$ almost everywhere. Finally, $v(P,\Phi)=v(C,\|\cdot\|_{\infty})$ implies $(P,\Phi)$ to be locally isometric almost everywhere to $\left(\widetilde{C},\|\cdot\|_{\infty}\right)\supset \left(C,\|\cdot\|_{\infty}\right)$ with $\widetilde{C}\setminus C$ negligible in $\R^{n}$.
\end{rmq}

\begin{exs} \label{ex:besicovitch}\ 
\begin{itemize}
\item For $v=v_{BH}$, it gives the sharp inequality:
\begin{equation*}
v_{BH}(P,\Phi) \geq \frac{b_{n}}{2^{n}}\prod_{i=1}^{n}d_{i}
\end{equation*}
where $b_{n}=\frac{\pi^{\frac{n}{2}}}{\Gamma\left(n+\frac{1}{2}\right)}$ designates the volume of the standard Euclidean unit ball.
\item For $v=v_{HT}$, it gives the sharp inequality:
\begin{equation*}
v_{HT}(P,\Phi) \geq \frac{2^{n}}{n! b_{n}}\prod_{i=1}^{n}d_{i}.
\end{equation*}
\end{itemize}
\end{exs}

The symmetry of $d_{\Phi}$ is key to get the inequality (\ref{ineqbes}), thus we can't directly extend this proof to the asymmetric Finsler case. Nevertheless, in the case of the Holmes-Thompson volume, the Roger-Shepard inequality allows us to assert the following

\begin{prop} \label{asym}
Let $P\subset \R^{n}$ be a $n$-dimensionnal parallelotope endowed with an asymmetric continuous Finsler metric $\Phi$. If $(F_i,G_i)$ (with $1\leq i\leq n$) denotes its pairs of opposite faces, then, 
\begin{equation*}
v_{HT}(P,\Phi) \geq \frac{n!}{(2n)!}\frac{2^{n}}{b_{n}}\prod_{i=1}^{n}\left(d_{\Phi}(F_{i},G_{i})+d_{\Phi}(G_{i},F_{i})\right).
\end{equation*}
\end{prop}
\begin{proof}
Following the proof of theorem 4.13 of \cite{Bal}, we consider the symmetrized Finsler metric $\Psi$ defined by
\begin{equation*}
\forall u\in TP,\quad \Psi(u):= \Phi(u)+\Phi(-u)
\end{equation*}
so that, for all curve $\gamma$,
\begin{equation*}
\quad \longueur_{\Psi}(\gamma) = \longueur_{\Phi}(\gamma) + \longueur_{\Phi}(\check{\gamma}),
\end{equation*}
where $\check{\gamma}$ designates the time-reversed curve. Thus, for all $x,y\in P$, $d_{\Psi}(x,y)\geq d_{\Phi}(x,y)+d_{\Phi}(y,x)$, hence
\begin{equation*}
\forall i,\quad  d_{\Psi}(F_{i},G_{i})\geq d_{\Phi}(F_{i},G_{i})+d_{\Phi}(G_{i},F_{i}).
\end{equation*}

On the other hand, at every $p\in P$, $B_{\Psi_{p}}=B_{\Phi_{p}}-B_{\Phi_{p}}$, thus, applying the Rogers-Shepard inequality at every cotangent space we have that
\begin{equation*}
v_{HT}(P,\Phi) \geq \frac{(n!)^{2}}{(2n)!}v_{HT}(P,\Psi).
\end{equation*}
The inequality then follows from proposition \ref{prop:besicovitch} applied to $(P,\Psi)$.
\end{proof}

\begin{rmq}
We can't hope such an inequality for the Busemann-Hausdorff volume in the asymmetric case. Here $d_{i}$ will designate $\min(d(F_{i},G_{i}),d(G_{i},F_{i}))$. To see it in $\R^{2}$, let take $P = [0,1]^2$ and let define an asymmetric norm $\Phi$ on $\R^{2}$ by its unit ball $B_{\Phi}$. Let $a=\left(-\frac{1}{2},0\right)$, $b=a+h\left(\frac{2}{3},1\right)$ and $c=a+h\left(\frac{2}{3},-1\right)$ where $h>\frac{3}{2}$; we define $B_{\Phi}$ as the triangle $abc$. As $h$ tends to infinity, $d_{1} \sim \frac{1}{h}$, $d_{2} = 1$ and $v_{BH}(P,\Phi) = \frac{3\pi}{2h^{2}}$, thus
\begin{equation*}
\frac{v_{BH}(P,\Phi)}{d_{1}d_{2}} \underset{h\to+\infty}{\longrightarrow} 0.
\end{equation*}

However, in the asymmetric flat case, we still have the weaker (sharp) inequality:
\begin{equation} \label{eq:flatBH}
v_{BH}(P,\Phi) \geq \frac{b_{n}}{2^{n}}\left(\min_{1\leq i\leq n}d_{i}\right)^{n}.
\end{equation}
As a matter of fact, taking $P=[0,1]^{n}$ without loss of generality,
\begin{equation*}
d :=\min_{1\leq i\leq n}d_{i} = \inf \{ \alpha>0,\ \alpha B_{\Phi}\cap\partial [-1,1]^{n} \neq \emptyset \},
\end{equation*}
thus for all $\alpha < d$, $\alpha B_{\Phi} \subset [-1,1]^{n}$, so $|dB_{\Phi}| \leq 2^{n}$ (where $|\cdot |$ designates the standard Lebesgue measure of $\R^{n}$) which is equivalent to (\ref{eq:flatBH}).

We can also show, with some duality, the Holmes-Thompson analogous of this last inequality: for all flat metric
\begin{equation*}
v_{HT}(P,\Phi) \geq \frac{2^{n}}{n!b_{n}}\left(\min_{1\leq i\leq n}d_{i}\right)^{n}.
\end{equation*}
As a matter of fact, with the last notations, for all $\alpha > 0$,
\[
\alpha B_{\Phi}\cap\partial [-1,1]^{n} \neq \emptyset \Leftrightarrow \exists i,\ \alpha B_{\Phi}\cap\{x,\ \la e_{i},x\ra =\pm 1 \}\neq \emptyset
\]
\[
\Leftrightarrow \exists i,\ e_{i}\not\in (\alpha B_{\Phi})^{*}\ \text{ or } -e_{i}\not\in (\alpha B_{\Phi})^{*}
\]
where the $e_{j}$ are the canonical base of $\R^{n}$. Thus for all $\alpha < d$, $(\alpha B_{\Phi})^{*} \supset B_{\|\cdot\|_{1}}$ the convex hull of the $\pm e_{j}$, so $|(dB_{\Phi})^{*}|\geq \frac{2^{n}}{n!}$.
\end{rmq}

\subsection{Finsler Papasoglu's lemma}

We can now extend the original proof of Papasoglu to Finsler 2-spheres.

\begin{prop}
Let $(\mathbb{S}^{2},\Phi)$ be a reversible Finsler two-sphere and let $\mathcal{A}$ be any Finsler volume and $c>0$ such as $\mathcal{A}([0,d_{1}]\times [0,d_{2}],\|\cdot\|_{\infty}) = cd_{1}d_{2}$ for all $d_{i}>0$. Then for any $\varepsilon>0$ there exists a closed curve $\gamma$ dividing $(\mathbb{S}^{2},\Phi)$ into two disks $D_{1}$ and $D_{2}$ of area at least $\frac{\mathcal{A}(\mathbb{S}^{2},\Phi)}{4}$ and whose length satisfies
\begin{equation}\label{eq:fpap}
\longueur (\gamma ) \leq 2\sqrt{\frac{3}{c}\mathcal{A}(\mathbb{S}^{2},\Phi)} + \varepsilon.
\end{equation}
\end{prop}

\begin{proof}
Lets having the same approach as the Riemannian proof, taking $\gamma\in \Gamma$ a curve of length $\longueur (\gamma)<L+\varepsilon$ dividing $\mathbb{S}^{2}$ on two disk $D_{1}$ and $D_{2}$ with the same conditions.

Let divide $\gamma$ on 4 curves $\gamma = \alpha_{1}\cup\alpha_{2}\cup\alpha_{3}\cup\alpha_{4}$ of the same length $\frac{\longueur(\gamma)}{4}$. As there is no $\varepsilon$-shortcut, we have got that
\begin{equation*}
d_{\Phi}(\alpha_{1},\alpha_{3}),d_{\Phi}(\alpha_{2},\alpha_{4}) \geq \frac{\longueur(\gamma)}{4} -\varepsilon.
\end{equation*}
Hence, by Besicovitch's lemma and the example \ref{ex:besicovitch},
\begin{equation*}
\mathcal{A} (D_{1}) \geq c\frac{(\longueur(\gamma)-4\varepsilon)^{2}}{16}
\end{equation*}
But $\mathcal{A} (D_{1})\leq \frac{3}{4}\mathcal{A} (\mathbb{S}^{2},\Phi)$, thus
\begin{equation*}
\longueur (\gamma ) \leq 2\sqrt{\frac{3}{c}\mathcal{A}(\mathbb{S}^{2},\Phi)} + 4\varepsilon.
\end{equation*}

\begin{figure}
\begin{center}
\includegraphics[scale =0.8]{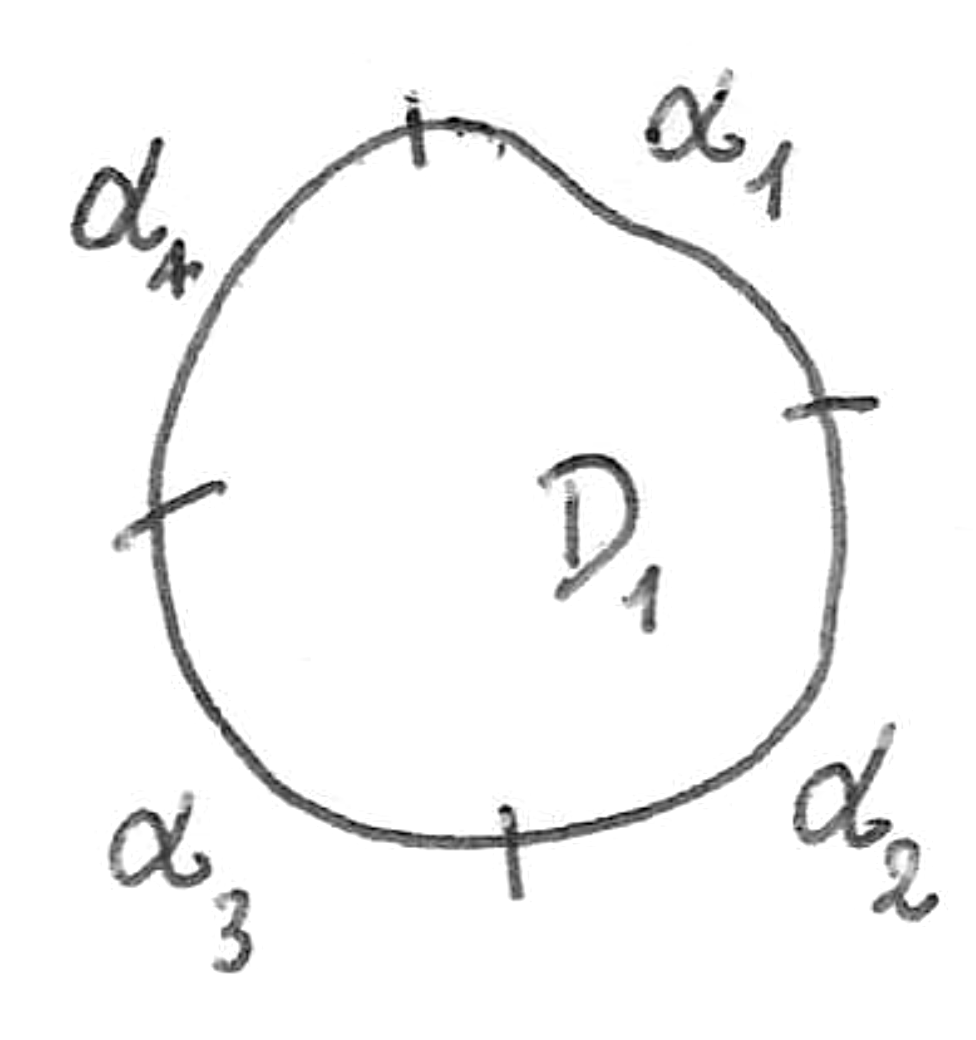}
\caption{Cutting $\gamma = \partial D_{1}$ in 4 curves of the same length}
\label{fig:papasoglu} 
\end{center}
\end{figure}
\end{proof}

\begin{exs}\ 
\begin{itemize}
\item For $\mathcal{A}=v_{BH}$, $c=\frac{\pi}{4}$ and (\ref{eq:fpap}) becomes:
\begin{equation*}
\longueur (\gamma ) \leq 4\sqrt{\frac{3}{\pi}\mathcal{A}(\mathbb{S}^{2},\Phi)} + \varepsilon.
\end{equation*}

\item For $\mathcal{A}=v_{HT}$, $c=\frac{2}{\pi}$ and (\ref{eq:fpap}) becomes:
\begin{equation*}
\longueur (\gamma ) \leq \sqrt{6\pi\mathcal{A}(\mathbb{S}^{2},\Phi)} + \varepsilon.
\end{equation*}

\end{itemize}
\end{exs}

Nevertheless, the proof of proposition \ref{papRiem2} gives a better estimate in these two special cases:
\begin{prop}
Let $(\mathbb{S}^{2},\Phi)$ be a reversible Finsler two-sphere, then  for $\mathcal{A}=v_{HT}$ or $v_{BH}$, there exists $\gamma$ such that,
\begin{equation*}
\longueur (\gamma ) \leq \sqrt{\frac{3\pi}{2}\mathcal{A}(\mathbb{S}^{2},g)} + \varepsilon,
\end{equation*}
with the same hypothesis on $\gamma$ and $\varepsilon$ as in the previous proposition.
\end{prop}

\begin{proof}
According to Ivanov's theorem 3 and 4 of \cite{Ivanov2011}, Pu's systolic inequality remains true for these two measures and any reversible Finsler metric $\Phi$. Thus, proof of proposition \ref{papRiem2} remains valid in this case.
\end{proof}

\begin{rmq}
the optimality issue discussed in remark \ref{optimal} still applies in the Busemann-Hausdorff case, according to Ivanov's theorem.
\end{rmq}

\bibliographystyle{alpha}
\bibliography{biblioM1} 

\textsc{S. Allais, École Normale Supérieure de Lyon, 69007 Lyon, France.}

\emph{E-mail address:} \texttt{simon.allais@ens-lyon.fr}
\end{document}